%% file: main.tex
\documentclass[12pt]{amsart}
\input{preamble}

\begin{document}

\input{cover}

\begin{abstract}
\input{abstract}
\end{abstract}

\maketitle

\input{paper_submit}
\bibliography{references}{}
\bibliographystyle{amsplain}

\end{document}

%% file: preamble.tex
\usepackage{setspace}
\usepackage{verbatim}

\usepackage{amsfonts,fancyhdr,ifthen,bm}
\usepackage{amscd,amssymb} 
\usepackage{times}
\usepackage{graphicx}
\usepackage{color}
\usepackage{epsfig}
\usepackage{mathrsfs}
\usepackage{paralist}
\usepackage{comment}

\usepackage[margin=1.5in]{geometry}

\newtheorem{thm}{Theorem}[section]
\newtheorem*{thm*}{Theorem}
\newtheorem{prop}[thm]{Proposition}

\newtheorem*{cor*}{Corollary}
\newtheorem{obs}[thm]{Observation}
\newtheorem{lem}[thm]{Lemma}
\newtheorem*{lem*}{Lemma}

\newtheorem*{oquest*}{Open Question}

\theoremstyle{remark}
\newtheorem{rmk}[thm]{Remark}

\theoremstyle{remark}
\newtheorem*{rmk*}{Remark}

\theoremstyle{definition}
\newtheorem{defn}[thm]{Definition}

\theoremstyle{definition}
\newtheorem*{defn*}{Definition}

\theoremstyle{definition}
\newtheorem{ex}[thm]{Example}

\numberwithin{equation}{section}

\newcommand{\OO}{\mathscr{O}}
\DeclareMathOperator{\FFF}{\mathbb{F}}



%% file: cover.tex


\title{
Irreducible canonical representations in positive characteristic
}
\author{
Benjamin Gunby, Alexander Smith and Allen Yuan
}

\thanks{This research was supported by an REU at Emory University under the mentorship of David Zureick-Brown.  We would also like to thank Ken Ono, Rachel Pries, and David Yang for advice and helpful comments.}

\date{\today}

%% file: abstract.tex
For $X$ a curve over a field of positive characteristic, we investigate when the canonical representation of $\text{Aut}(X)$ on $H^0(X, \Omega_X)$ is irreducible.  Any curve with an irreducible canonical representation must either be superspecial or ordinary. Having a small automorphism group is an obstruction to having irreducible canonical representation; with this motivation, the bulk of the paper is spent bounding the size of automorphism groups of superspecial and ordinary curves. After proving that all automorphisms of an $\FFF_{q^2}$-maximal curve are defined over $\FFF_{q^2}$, we find all superspecial curves with $g > 82$ having an irreducible representation. In the ordinary case, we provide a bound on the size of the automorphism group of an ordinary curve that improves on a result of Nakajima.

%% file: paper_submit.tex
\input{intro}

\input{prelim}
\input{superspecial}

\input{thm2case123}

\input{thm2case4a}

\input{thm2case4b}

%% file: intro.tex
\section{Introduction}
\label{sect:Intro} 	
\

Given a complete nonsingular curve $X$ of genus $g\geq 2$, the finite group $G:=\text{Aut}(X)$ has a natural action on the $g$-dimensional $k$-vector space $H^0(X, \Omega_X)$, known as the \emph{canonical representation}.  It is natural to ask when this representation is irreducible.  In characteristic zero, irreducibility of the canonical representation implies that $g^2 \leq |G|$, and combining this with the Hurwitz bound of $|G|\leq 84(g-1)$, one can observe that the genus of $X$ is bounded.  In fact, Breuer \cite{Breu00} shows that the maximal genus of a Riemann surface with irreducible canonical representation is 14.  


In characteristic $p$, the picture is more subtle when $p$ divides $|G|$.  The Hurwitz bound of $84(g-1)$ may no longer hold due to the possibility of wild ramification in the Riemann-Hurwitz formula.  It is known that when $2\leq g \leq p-2$, the Hurwitz bound holds with one exception given by Roquette \cite{Roqu70}: the hyperelliptic curve $y^2 = x^p-x$, which has genus $\frac{p-1}{2}$ and $2p(p^2-1)$ automorphisms.  For the general case, Henn \cite{Henn78} classifies curves for which $|G|\geq 8g^3$, but the problem of classifying curves with more than $84(g-1)$ automorphisms is not well understood.  Hence, in characteristic $p$, the bound of $g \le 82$ no longer applies. Indeed, Hortsch \cite{Hort12} shows that the Roquette curve has irreducible canonical representation, providing an example of arbitrarily high genus curves with irreducible canonical representation.  Dummigan \cite{Dumm95} showed that the Fermat curve (which is equivalent to the Hermitian or Drinfeld curve after suitable change of coordinates) $$x^{p+1}+y^{p+1}+z^{p+1} = 0$$ in characteristic $p$ is another such example.

It is natural to ask whether for fixed $p$, there exist characteristic $p$ curves $X$ of arbitrarily high genus $g$ with irreducible canonical representation.  Via Observation \ref{obs:cases}, this question splits into two cases depending on whether the curve is \emph{superspecial} (i.e., the Frobenius acts as $0$ on $H^1(X, \OO_X)$) or \emph{ordinary} (i.e., the $p$-rank of $X$ equals $g$).  

In the superspecial case, we determine that the only examples of curves with irreducible canonical representation are isomorphic to either the Roquette curve or the Fermat curve given above. To do this, we show that we may work with $\FFF_{p^2}$ maximal curves and subsequently prove a strict condition for a high genus $\FFF_{p^2}$-maximal curve to have $|\text{Aut}(X)| > g^2$. Since $|\text{Aut}(X)| > g^2$ is a necessary condition for irreducibility, this will give the result. To prove the condition for $\FFF_{p^2}$ maximal curves to have many automorphisms, which is given as Theorem \ref{thm:max}, we rely on the new result that all automorphisms of an $\FFF_{q^2}$ maximal curve are defined over $\FFF_{q^2}$, which is given as Theorem \ref{thm:Fq2_auto}.


We then reduce the question to considering ordinary curves, where progress is harder.  Nakajima \cite{Naka87} bounds the automorphism group of an ordinary curve by $84g(g-1)$.  In this paper, we prove the following stronger result:

\begin{thm}\label{main}
There exists a constant $c=c(p)$ such that any nice ordinary curve $X$ over a field of characteristic $p$ with genus $g_X>c$ satisfies the inequality $$|\textup{Aut}(X)| \leq 6\left(g_X^2+12\sqrt{21}g_X^{\frac{3}{2}}\right).$$
\end{thm}
\begin{rmk}
From our proof, we may take $c(p)$ to be on the order of $p^2$.
\end{rmk}

This does not yet imply reducibility of the canonical representation, but we have the following.

\begin{rmk} An unpublished result of Guralnick and Zieve \cite{GuraZiev} states that for any prime $p$ there is a positive constant $c_p$ so that, if $X$ is an ordinary curve of genus $g > 1$ over an algebraically closed field of characteristic $p$, the group of automorphisms of $X$ has order bounded by $c_pg^{8/5}$. Together with the superspecial results, this would imply that for a fixed characteristic, there do not exist arbitrarily high genus curves with irreducible canonical representation. We hope that the eventual published work will give even stronger ways of characterizing ordinary curves with irreducible canonical representations. \end{rmk}

%% file: prelim.tex
\section{Preliminaries}
\label{sect:Prelim} 	
\

In this section, we present notations and basic techniques that will be used throughout the paper.  Let $k$ be an algebraically closed field of characteristic $p>0$.  Throughout, a \emph{nice} curve is a complete nonsingular curve over $k$.  For a curve $X$, we shall use $g_X, \gamma_X,$ and $\text{Aut}(X)$ to denote its genus, $p$-rank of Jacobian, and automorphism group, respectively.  For any two curves $X$ and $Y$, $\pi_{X/Y}$ will denote a (branched) covering map $X\to Y$ if there is no ambiguity with respect to the map in question.  

\subsection{Ramification Groups}

Given a curve $X$, a finite subgroup $G\subset \text{Aut}(X)$, and a point $P\in X$, define the ramification groups $G_i(P)$ for $i\geq 0$ as follows:
\begin{equation}
\label{blank} G_0(P) = \{ \sigma \in G | \sigma P = P\}
\end{equation}
and  
\begin{equation}
\label{Gi} G_i(P) = \{ \sigma \in G_0(P) | \text{ord}_P(\sigma \pi_P - \pi_P) \geq i+1\}
\end{equation}
 for $i>0$, where $\pi_P$ is a uniformizer at $P$ and $\text{ord}_P$ denotes the order of the zero at $P$.  

We will use the following fact about the ramification groups, which can be found in \cite[Sections 2,3]{Naka87}:

\begin{prop}\label{ramprop}
Let $X$ be a nice \emph{ordinary} curve (i.e. $g_x = \gamma_x$), $G$ be a subgroup of $\text{Aut}(X)$, and $P\in X$ be any point.  Then the following are true:
\begin{enumerate}
\item $G_1(P)$ is a normal Sylow $p$-subgroup of $G_0(P)$.
\item $G_0(P)/G_1(P)$ is a cyclic group. 
\item $G_2(P) = \{ 1\}$. \end{enumerate}
\end{prop}

\subsection{Galois Coverings of Curves}

Let $\pi \colon X\to Y$ be a Galois covering of curves with Galois group $G$.  For instance, $Y$ may be the quotient of $X$ by a finite subgroup $G$ of automorphisms.  For $Q\in Y$ and any point $P\in \pi^{-1}(Q)$, let $e_Q$ denote the ramification index of $\pi$ at $P$, and $d_Q := \sum_{i=0}^{\infty}(|G_i(P)|-1)$.  Note that $e_Q$ and $d_Q$ do not depend on the choice of $P$.  We also have that $d_Q \geq e_Q-1$ with equality if and only if $Q$ is tamely ramified.  The relationship between the genera of $X$ and $Y$ is given by the Riemann-Hurwitz formula:

\begin{equation}
\label{RH} \frac{2g_X-2}{|G|} = 2g_Y -2 + \sum_{Q\in Y} \frac{d_Q}{e_Q}.
\end{equation}

A similar formula, known as the {D}euring-\v{S}afarevi{\v{c}} formula, relates the $p$-ranks of $X$ and $Z = X/H$ when $H \subseteq G$ is a $p$-group:

\begin{equation}
\label{DS} \frac{\gamma_X -1}{|H|} = \gamma_Z -1 + \sum_{Q\in Z}(1-e_Q^{-1}).
\end{equation}

\subsection{Frobenius and $p$-rank}

The idea of studying the canonical representation via Frobenius has appeared in \cite{Naka85}.  It may be adapted to our situation as follows.  The action of Frobenius on a curve $X$ gives a natural action of Frobenius on $H^1(X, \OO_X)$.  In fact, for particular curves, this action can be explicitly computed on a basis of $H^1(X, \OO_X)$ by using a \v{C}ech cover of $X$.  In addition, by Serre duality, $H^1(X,\OO_X)$ has a natural $k[\text{Aut}(X)]$-module structure as the dual representation of the canonical representation.  In fact, it is not difficult to see that the Frobenius commutes with this module structure.  Since the Frobenius map does not respect scalar multiplication, it is not quite a $k[\text{Aut}(X)]$-module homomorphism, but we may still observe the following:



\begin{obs}
If $X$ has irreducible canonical representation, then the $\mathbb{F}_p$ vector space map $F\colon H^1(X, \OO_X) \rightarrow H^1(X, \OO_X)$ given by Frobenius is either injective or zero.
\end{obs}
\begin{proof}
The kernel of this map is a $k$-vector space which is invariant under $\text{Aut}(X)$ since automorphisms (on the function field) commute with $F$.  Thus, the kernel is a $k[\text{Aut}(X)]$-module, so if it is a proper nonzero submodule, then the canonical representation has a proper subrepresentation, contradicting the hypothesis.  Hence, the kernel must either equal $H^1(X, \OO_X)$ or zero.
\end{proof}

In the literature, the matrix of the action of $F$ on $H^1(X, \OO_X)$ is referred to as the Hasse-Witt matrix, with the corresponding dual action on $H^0(X, \Omega_X)$ being known as the Cartier operator.  If the action is injective, it is immediate that the $p$-rank of $X$ is equal to its maximum possible value, $g$; such curves are called \emph{ordinary}.  On the other hand, if the action is zero, the curve is known as \emph{superspecial}. The typical definition of a superspecial curve is a nice curve whose corresponding Jacobian variety is a product of supersingular elliptic curves.  However, Nygaard \cite[Theorem 4.1]{Nyga81} shows that this condition is equivalent to the Cartier operator being zero. With this extra notation, we can restate the above observation:

\begin{obs}
\label{obs:cases}
Let $X$ be a nice curve.  If the canonical representation of $X$ is irreducible, then $X$ is either superspecial or ordinary.
\end{obs}

\begin{ex}
Let $X$ be the genus $2$ curve given by $y^2 = x^5 - x$. Considered as a Riemann surface, this curve is called the Bolza surface and has $48$ automorphisms, the most of any genus two curve \cite{Katz14}. In our case, we only care about the automorphisms $\alpha(x, y) = \left(\frac{1}{x}, \frac{y}{x^3}\right)$ and $\beta(x, y) = (-x, iy)$, where $i^2 = -1$.

Let $U = X\backslash \{0 \}$ and $V = X \backslash \{ \infty \}$. $\{U, V\}$ is a Cech cover for the curve, so we have
\[H^1(X, \OO_X) = \Gamma(U \cap V, \OO_X) / \left(\Gamma(U, \OO_X) \oplus \Gamma(V, \OO_X)\right) \]
By Riemann-Roch, there are precisely two orders $l$ so that there is no rational function regular on $U$ with a pole of order $l$ at the origin. From the function $1/x$, we see $2$ is not one of these orders, leaving just $1$ and $3$.  From the \v{C}ech decomposition, this implies that $y/x$ and $y/x^2$ form a basis for $H^1(X, \OO_X)$.

Using these automorphisms, and assuming that we are over a field of characteristic $p \ne 0, 2$, it is easy to see that the canonical representation of this curve is irreducible. For $\alpha$ has distinct eigenvectors $\frac{y}{x} \pm \frac{y}{x^2}$ corresponding to eigenvalues $\pm 1$, and it is easy to see that neither of these are fixed by $\beta$.

By the above criteria, we then know that this curve is either ordinary or superspecial. To determine which it is, we apply Frobenius to the two basis elements:

\[F\left(\frac{y}{x}\right) = \frac{y^p}{x^p} = \frac{y}{x^p}\left(x^5 - x\right)^{\frac{1}{2}(p-1)} = \frac{y}{x^{\frac{1}{2}(p+1)}}\left(x^4 - 1\right)^{\frac{1}{2}(p-1)}.\]
For $i \ne 1, 2$, $y/x^i$ is trivial in the cohomology group, so almost every term vanishes. Indeed, we get
\[F\left(\frac{y}{x}\right) = \begin{cases} K_1 \frac{y}{x} &\mbox{if } p \equiv 1 \\ 
K_2 \frac{y}{x^2}& \mbox{if } p \equiv 3 \\
0 & \mbox{if } p \equiv 5, 7 \end{cases} \pmod 8\]
where $K_1, K_2$ are nonzero binomial coefficients, and
\[F\left(\frac{y}{x^2}\right) = \begin{cases} K_3 \frac{y}{x} &\mbox{if } p \equiv 3 \\ 
K_4 \frac{y}{x^2}& \mbox{if } p \equiv 1 \\
0 & \mbox{if } p \equiv 5, 7 \end{cases} \pmod 8.\]
From this, we see that the Bolza surface is ordinary if $p \equiv 1, 3 \pmod 8$ and is superspecial if $p \equiv 5, 7 \pmod 8$.
\end{ex}

%% file: superspecial.tex
\section{Superspecial curves}
\begin{prop}
If $X$ is a superspecial curve in a field of characteristic $p$, then $g \le \frac{1}{2}p(p-1)$. If $X$ is additionally hyperelliptic, then $g \le \frac{1}{2}(p-1)$.
\end{prop}
\begin{proof}
This is Theorem 1.1 of \cite{Eked87}.
\end{proof}
Both of these bounds are sharp, with the Hermitian curve $x^p - x = y^{p+1}$ giving an example of the equality case for the first bound and the hyperelliptic curve $y^2 = x^p - x$ giving an example of the equality case for the second.

These two examples can be reused for the following proposition:
\begin{prop} \label{dh}
The maximal genus of a superspecial curve with irreducible canonical representation is $\frac{1}{2}p(p-1)$, a bound attained by the Hermitian curve given by $x^p - x = y^{p+1}$. The maximal genus of a hyperelliptic superspecial curve with irreducible canonical representation is $\frac{1}{2}(p-1)$, a bound attained by the curve given by $y^2 = x^p - x$.
\end{prop}
\begin{proof}
Dummigan \cite{Dumm95} and Hortsch \cite{Hort12} prove that the two desired curves have irreducible canonical representation, so we are done by the previous proposition.  
\end{proof}
In fact, these are the only two infinite families of superspecial curves with irreducible canonical representation, in the sense of the following theorem:

\begin{thm}
\label{thm:sup_irred}
If $X$ is a superspecial curve of genus $g > 82$ over an algebraically closed field of characteristic $p$, then $X$ has an irreducible canonical representation if and only if $X$ is isomorphic to the curve given by $y^2 = x^p - x$ or to the curve given by $y^{p+1} = x^p - x$.
\end{thm}

To prove this result, we transmute the problem to looking at maximal and minimal curves, which we define now.

\subsection{Maximal and minimal curves}
\begin{thm}
(Weil conjecture for curves) If $X$ is a nonsingular curve of genus $g$ defined over $\FFF_{q}$, then there are $2g$ complex constants $\omega_1, \dots, \omega_{2g}$ of magnitude $q^{1/2}$ so that $N_n$, the number of $\FFF_{q^n}$-rational points on the curve, satisfies
\[N_n = q^n + 1 - \sum_{i=1}^{2g} \omega_i^n.\]
In particular, $q^n - 2gq^{n/2}  + 1 \le N_n \le q^n + 2gq^{n/2} + 1$.
\end{thm}
\begin{defn}
A curve defined over $\FFF_{q^2}$ is called \emph{$\FFF_{q^2}$ maximal} if the number of $\FFF_{q^2}$ rational points is $q^2 + 2gq + 1$, and is called \emph{$\FFF_{q^2}$ minimal} if the number of $\FFF_{q^2}$ rational points is $q^2 - 2gq + 1$.
\end{defn}

The relevance of these curves is immediate from the following proposition:
\begin{prop}
\label{prop:sup_is_mm}
A curve $X$ defined over the algebraic closure of $\FFF_{p}$ is superspecial if and only if it is isomorphic to an $\FFF_{p^2}$-maximal curve or an $\FFF_{p^2}$-minimal curve.
\end{prop}
\begin{proof}
This is a consequence of Ekedahl's work in \cite{Eked87} and is proved as Theorem 2.6 of \cite{Kaze13}.
\end{proof}

The $L$-functions corresponding to maximal and minimal curves are very simple, equaling $(qt +1)^{2g}$ for maximal curves and $(qt - 1)^{2g}$ for minimal curves. As was first noted in \cite{Fuhr97} (and extended to minimal curves in \cite{Vian05}), this is enough to imply the following result, which is known as the \emph{fundamental equation}
\begin{prop}
\label{prop:fund_eq}
Suppose $X$ is an $\FFF_{q^2}$ maximal curve over an algebraically closed field of characteristic $p > 0$. If $P_0$ is an $\FFF_{q^2}$ rational point, then we have the linear equivalence
\[(q+1)P_0 \sim qP + F(P)\]
where $F$ denotes the degree $q^2$ Frobenius map and $P$ is some other closed point of the curve. If $X$ is instead assumed to be an $\FFF_{q^2}$ minimal curve, we have
\[(q-1)P_0 \sim qP - F(P).\]
\end{prop}

\subsection{Automorphisms of maximal curves}

\begin{prop}
Let $X$ be an $\mathbb{F}_{q^2}$ maximal curve of genus at least two. Then three distinct points $P, Q, R$ satisfy 
\begin{equation}
\label{eq:PQR}
(q+1)P \sim (q+1)Q \sim (q+1)R
\end{equation}
if and only if all three points are $\FFF_{q^2}$ rational.
\end{prop}
\begin{rmk}
This proposition and proof also works for minimal curves, being based off the rational point relation $(q-1)P \sim (q-1)Q \sim (q-1)R$ instead of  \eqref{eq:PQR}.
\end{rmk}
\begin{proof}
The fundamental equation gives that, if $P, Q, R$ are all $\FFF_{q^2}$ rational, they satisfy \eqref{eq:PQR}. Conversely, suppose \eqref{eq:PQR} were satisfied for three distinct points $P, Q, R$. We note that if one of these points is $\FFF_{q^2}$ rational, then all three are. For supposing $P$ were $\FFF_{q^2}$ rational, we can write
\[qQ + F(Q) \sim (q+1)P \sim (q+1)Q\]
where $F$ is the $q^2$ Frobenius map. Then $Q \sim F(Q)$, and hence $Q = F(Q)$.

The fundamental equation gives $qP + F(P) \sim qQ + F(Q)$, and subtracting from \eqref{eq:PQR} then gives $P + F(Q) \sim Q + F(P)$. If this system is not base point free, given that $P \ne Q$, we find $P = F(P)$, so the points are $\FFF_{q^2}$ rational. Then we assume $|P + F(Q)|$ is a base-point free linear system, so $X$ must be hyperelliptic.

This same argument shows that $|P + F(R)|$ is a base-point free linear system. Both these divisors are of degree two and dimension one, so since the genus of the curve is at least two, the divisors must be the same (see \cite[p.216]{Hirs13}). Then $P + F(Q) \sim P + F(R)$, so $Q = R$. This contradicts the three points being distinct, proving the proposition.
\end{proof}

\begin{thm}
\label{thm:Fq2_auto}
An automorphism of a $\mathbb{F}_{q^2}$ maximal or minimal curve of genus at least two fixes the set of $\mathbb{F}_{q^2}$ rational points. More generally, automorphisms of maximal or minimal curves always commute with the $q^2$ Frobenius map, and are hence defined as $\FFF_{q^2}$ maps.
\end{thm}
\begin{proof}
Take $X$ to be such a maximal curve, and suppose there were an automorphism $\sigma$ that maps an $\FFF_{q^2}$ rational point $P_1$ to a non-$\FFF_{q^2}$-rational point. Choose three distinct $\FFF_{q^2}$ rational points $P_1, P_2, P_3$. Then $(q+1)P_1 \sim (q+1)P_2 \sim (q+1)P_3$, so
\[(q+1)(\sigma P_1) \sim (q+1)(\sigma P_2) \sim (q+1)(\sigma P_3). \]
But $\sigma P_1$ being not rational contradicts the previous proposition. Then any automorphism fixes the rational points.

This implies that automorphisms commute with the Frobenius map. For take $P$ an arbitrary point, $\sigma$ an arbitrary automorphism, and $P_0$ any rational point. Then
\[F(\sigma P) \sim (q+1)P_0 - q \sigma P \sim (q+1) \sigma P_0 - q \sigma P \sim \sigma F(P).\]
This establishes the theorem in the maximal case.

The minimal curve case falls to an identical argument if there are at least three $\FFF_{q^2}$ rational points. If there are not, we must have $q^2 - 2gq + 1$ either equal to one or two. Since $q^2 - 2gq + 1$ is one mod $q$, it must equal one if it is less than three, with $g = q/2$. Then in this exceptional case, $q$ must be a power of two.

Let $X$ be an $\FFF_{q^2}$-minimal curve in characteristic $2$ with a single $\FFF_{q^2}$ rational point. Suppose there were an automorphism $\sigma$ of the curve that did not fix the unique $\FFF_{q^2}$-rational point $P_0$. Let $P_1 = \sigma^{-1} P_0$. Then the fundamental equation gives
\[(q-1)P_0 \sim q\sigma P_0  - F(\sigma P_0)\]
and
\[(q-1)\sigma P_0 \sim q\sigma P_1 - \sigma F(P_1) \sim q P_0 - \sigma F(P_1).\]
Then
\[F(\sigma P_0) + \sigma F(P_1) \sim P_0 + \sigma P_0.\]
But $P_0$ is neither $\sigma F(P_1)$ nor $F(\sigma P_0)$, so since the curve is not rational, $|P_0 + \sigma P_0|$ is base-point free. We similarly find that $|\sigma P_0 + \sigma^2 P_0|$ is base-point free, so again by the uniqueness of such a linear system, we have $\sigma P_0 + \sigma^2 P_0 \sim P_0 + \sigma P_0$, for $P_0 = \sigma^2 P_0$.

Next, since $X$ is an $\FFF_{q^4}$-minimal curve too, and since $q^4 - 2gq^2 + 1 = q^4 - q^3 + 1 \ge 3$, we can use what we have already proved to say that $\sigma$ fixes the set of $\FFF_{q^4}$-rational points, of which there are an odd number. Write the order of $\sigma$ as $2^l \cdot r$ with $r$ odd. Then $\sigma^r$ still switches $P_0$ and $P_1$, but in addition has order a power of two. Since there are an odd number of $\FFF_{q^4}$ rational points, this implies that $\sigma^r$ fixes some $\FFF_{q^4}$-rational point. Call this point $P_2$. Then
\[(q-1)P_0 \sim qP_2 - F(P_2)\]
and
\[(q-1)\sigma^r P_0 \sim q \sigma^r P_2 - \sigma^r F(P_2)\]
for
\[(q-1) P_1 \sim q P_2 - \sigma^r F(P_2).\]
Then
\[(q-1)P_0 + F(P_2) \sim (q-1)P_1 + \sigma^r F(P_2)\]
so that
\[ F(P_2) + P_1  \sim \sigma^r F(P_2) + F(P_1).\]
This must be another base-point free divisor of degree two. From the uniqueness of such a divisor, we have $F(P_2) + P_1 \sim P_0 + P_1$, for $F(P_2) \sim P_0$. This is impossible, implying that the automorphism group of $X$ fixes the unique $\FFF_{q^2}$ rational point. Then the automorphism group of an $\FFF_{q^2}$-minimal curve fixes the $\FFF_{q^2}$ rational points. That automorphisms commute with Frobenius is proved exactly as it was before.
\end{proof}
\begin{rmk}
This result has been observed before for specific curves. In \cite{Gura12}, for instance, the automorphism group of a family of maximal curves is calculated, and it is observed that all automorphisms are defined over $\FFF_{q^2}$.
\end{rmk}

\subsection{Consequences of Theorem \ref{thm:Fq2_auto}}
Theorem \ref{thm:Fq2_auto} is a very strong result for understanding the structure of the automorphism groups of maximal and minimal curves, telling us that if $X$ is $\FFF_{q^2}$ maximal or minimal, the group of automorphisms fixes the set of $\FFF_{q^{2n}}$ rational points for any $n$, a set whose cardinality we already knew. This can be exploited.

The first proposition does not exploit Theorem \ref{thm:Fq2_auto}, but it is important for understanding how wild orbits behave.
\begin{prop}
\label{prop:max_pgrp}
Let $X$ be an $\mathbb{F}_{q^2}$ maximal or minimal curve of genus $g\geq 2$. Take $H$ to be a $p$-subgroup of automorphisms. Then $H$ fixes a unique point and acts freely on all other points.
\end{prop}
\begin{proof}
Maximal and minimal curves are known to be supersingular (see \cite{Kaze13}), so the $p$-rank of a maximal or minimal curve is always zero. Write $|H| = p^k$. Then from the {D}euring-\v{S}afarevi{\v{c}} formula, we have
\[\frac{-1}{p^k} = -1 + \sum_{Q \in Z} \left(1 - e_{Q}^{-1}\right)\]
where $Z = X/H$. We immediately conclude that the map ramifies with index $p^k$ at a single point.
\end{proof}
For a $p$-group $H$ of automorphisms of a maximal curve, note that this proposition implies $|H| \bigm| q^2 + 2gq$.

From here forward, we will use the term \emph{short orbit} to denote any non-free orbit of $G$ in $X$, that is, an orbit where each point has nontrivial stabilizer in $G$.

\begin{prop}
\label{prop:aut_divis_rough}

Suppose $X$ is a $\mathbb{F}_{q^2}$ maximal curve of genus $g\geq 2$, and suppose the automorphism group $G$ of $X$ has only free orbits outside of the $\mathbb{F}_{q^2}$ rational points. Then 
\[|G| \bigm| 2q^3\left(q^2 - 1\right)\left(q+1\right).\]

If $X$ is instead a $\mathbb{F}_{q^2}$ minimal curve of genus $g\geq 2$, and the automorphism group $G$ has only free orbits outside the $\mathbb{F}_{q^2}$ rational points, then
\[ |G| \bigm| 2q^3\left(q^2 - 1\right)\left(q-1\right).\]
\end{prop}

\begin{proof}
First suppose $X$ is maximal.  By Theorem \ref{thm:Fq2_auto}, we can say the set $S_{(a, b)}$ of $\mathbb{F}_{q^{2ab}}$ rational points that are not $\mathbb{F}_{q^{2a}}$ rational is closed under $G$. If $G$ has no short orbits outside the $\mathbb{F}_{q^2}$ rational points, then we know that $|G|$ will divide $|S_{(a, b)}|$, so for all $a, b \ge 1$,
\begin{equation}
\label{eq:only_lorbits}
|G| \bigm| q^{2ab} + (-1)^{ab - 1}2gq^{ab} - q^{2a} + (-1)^a2gq^a.
\end{equation}
Take $T$ to be the maximal divisor of $|G|$ not divisible by $p$. From the above expression, we get
\[T \bigm| \left(q^{a(b-1)} +(-1)^{a(b-1) - 1}\right)\left((-1)^{ab - 1}2g + q^{ab} + (-1)^{a(b-1)}q^a\right).\]
Write $m = a(b-1)$, and take $a, a'$ to be different factors of $m$. If $m$ is odd, we get
\begin{equation}
\label{eqn:max_m_odd_div}
T \bigm| \left(q^m + 1\right)\left(-2g + q^{m + a} - q^a\right)
\end{equation}
Subtracting this from the similar expression for $a'$ gives, for $m$ odd,
\begin{equation}
\label{eq:no_more_g}
T \bigm| \left(q^{2m} - 1\right)\left(q^{a - a'} - 1\right)
\end{equation}
From $m = 3$, we get $T \bigm| (q^6 - 1)(q^2 - 1)$. From $m = 5$, we get $T \bigm| (q^{10} - 1)(q^4 - 1)$. $1 + q^2 + q^4$ is coprime to $q^2 + 1$ and $1 + q^2 + q^4 + q^6 + q^8$, so we get $T \bigm| (q^2 -1)^2$. If $m = a(b-1)$ is even, we use the same process, first getting
\begin{equation}
\label{eqn:max_m_even_div}
T \bigm| \left(q^m - 1\right)\left((-1)^{a-1}2g + q^{m + a} + q^a\right).
\end{equation}
 If $a, a'$ are divisors of $m$, we get
\[T \bigm| \left(q^{2m} - 1\right)\left(q^{a - a'} - (-1)^{a - a'}\right).\]
From $m = 2$, $T \bigm| (q^4 - 1)(q + 1)$, and the GCD of $q^2 + 1$ and $q -1$ is two, so we have that $T$ is a factor of $2(q^2 -1)(q+1)$. The relation on $|G|$ then follows from Proposition \ref{prop:max_pgrp}.

In the minimal case, \eqref{eq:no_more_g} holds for all $m$. Taking the GCD for $m = 2, 3$ gives the result.
\end{proof}

\begin{prop}
\label{prop:aut_divis_fine}
Suppose $X$ is a $\mathbb{F}_{q^2}$ maximal curve of genus at least two, and suppose the automorphism group $G$ of $X$ has only free orbits outside of the $\mathbb{F}_{q^2}$ rational points. Then
\[ |G| \bigm| 2q^3(q +1)\cdot \textup{gcd}(2g - 2, q+1)\cdot \textup{gcd}(4g, q - 1).\]
For minimal curves, the relation is instead
\[ |G| \bigm| 2q^3(q -1)\cdot \textup{gcd}(2g - 2, q -1)\cdot \textup{gcd}(4g, q + 1).\]
\end{prop}
\begin{proof}
We start with the maximal case. From \eqref{eqn:max_m_odd_div},
\[\text{gcd}\left((q+1)^2, |G|\right) \bigm| \left(q^m + 1\right)\left(-2g + q^{m + a} - q^a\right)\]
for $m$ odd. Choose $m$ coprime to $q + 1$. Then we get
\[\text{gcd}\left((q+1)^2, |G|\right) \bigm| (q+1)\text{gcd}\left(q+1, -2g + q^{m+a} - q^a\right)\]
or
\[\text{gcd}\left((q+1)^2, |G|\right) \bigm|(q+1)\text{gcd}\left(q+1, 2g -2\right).\]
Next, we have
\[\text{gcd}\left(q-1, |G|\right) \bigm| \left(q^m + 1\right)\left(-2g + q^{m + a} - q^a\right)\]
for
\[\text{gcd}\left(q-1, |G|\right) \bigm| 4g.\]
Together with the previous proposition, these two relations imply the result for maximal curves. For minimal curves, we start instead with
\[ |G| \bigm| \left(q^{m+a} - q^a\right)\left(-2g + q^{m + a} + q^a\right)\]
to derive the other result.
\end{proof}

The next technical lemma is a direct application of Proposition \ref{prop:aut_divis_fine}, and allows us to bound the size of automorphism groups for some curves with relatively small genus.
\begin{lem}
\label{lem:fine_cor}
Suppose $X$ is a $\mathbb{F}_{p^2}$ maximal curve of genus at least two over a field with characteristic an odd prime, and suppose the automorphism group $G$ of $X$ has only free orbits outside of the $\mathbb{F}_{q^2}$ rational points. Suppose the genus of $X$ can be written as $g = \frac{1}{2}c(p-1) + dp$. Then
\[|G| \bigm| 16q^3(q+1)d(c + d + 1).\]
\end{lem}
\subsection{$\FFF_{p^2}$-maximal curves with many automorphisms}

In this section we prove the following theorem:
\begin{thm}
\label{thm:max}
Let $X$ be a $\FFF_{p^2}$ maximal curve of genus $g$ at least $2$, where $p > 7$, and let $|G|$ be the group of automorphisms of $X$. If $|G| > g^2$ and $|G| > 84(g-1)$, $X$ must be isomorphic to a curve of the form $x^m = y^p - y$, where $m > 1$ and $m \bigm| p+1$.
\end{thm}

As a first step, we determine the structure of a Sylow subgroup of $G$.

\begin{prop}\label{hermp^2}
Let $X$ be a $\FFF_{p^2}$ maximal curve of positive genus, and suppose $G$ has order divisible by $p$. Then we can write
\[g = \frac{1}{2}c(p-1) + dp.\]
Furthermore, unless $X$ is isomorphic to the Hermitian curve, we must have that $p^2$ does not divide $|G|$.
\end{prop}
\begin{proof}
Let $H$ be a Sylow $p$-subgroup of $G$, and let $H'$ be a subgroup of $H$ of size $p$. Per Proposition \ref{prop:max_pgrp}, $H'$ stabilizes a single point $P$. Then
\[\frac{2g-2}{p} = 2d - 2 + \frac{(c+2)(p-1)}{p}\]
where $i = c+1$ is the final index so $G_{i}(P)$ has order $p$ and $d$ is the genus of $X/H'$. This gives $g$ in the form $\frac{1}{2}c(p-1) + dp$.

Next, we recall that the Hermitian curve is the only maximal curve attaining the maximal genus $\frac{1}{2}p(p-1)$. Then, unless the curve $X$ is Hermitian, $p^2 + 2gp < p^3$. Then, if $H$ is not of order $p$, it is of order $p^2$. But then we can write
\[\frac{2g-2}{p^2} = 2f - 2 + \frac{1}{p^2}\left((c+2)(p^2 - 1) + bp(p-1)\right)\]
for some nonnegative integers $f, b, c$. We are here using the fact that, for $a_1, a_2 \ge 1$, if $G_{a_1 + 1}(P)$ is distinct from $G_{a_1}(P)$ and $G_{a_2+1}(P)$ is distinct from $G_{a_2}(P)$, then $a_1 - a_2 \equiv 0 \pmod p$, see \cite[Lemma 11.75(v)]{Hirs13}. Simplifying gives
\[2g = 2fp^2 + c(p^2 -1) + bp(p-1).\]
But $2g < p(p-1)$ for $X$ not the Hermitian curve, so we get $f = c = b = g = 0$ unless the curve is Hermitian.
\end{proof}

We will use the notation $g = \frac{1}{2}c(p-1) + dp$ throughout this section. In this notation, $d$ is the genus of the quotient curve from dividing by an order $p$ subgroup.

As is traditional when studying large automorphism groups (see \cite{Henn78} or \cite{Naka87}), we use the following lemma to split into cases. For this lemma, a \emph{wild} short orbit is an orbit where each point has nontrivial stabilizer of order dividing $p$, and a \emph{tame} short orbit is an orbit where each point has nontrivial stabilizer of order not dividing $p$.

\begin{lem}
\label{lem:worb_types}
Let $G$ be the automorphism group of a genus $g \ge 2$ curve $X$ defined over an algebraically closed field of characteristic $p$.  If $|G| > 84(g-1)$, then one of the following four cases applies.

(I) $p \ne 2$ and $G$ has precisely three short orbits in $X$, one wild and two tame, with each point in the tame orbits having an order-two stabilizer.

(II) $G$ has precisely two short orbits, both wild.

(III) $G$ has one wild short orbit and no other short orbits.

(IV) $G$ has one wild short orbit, one tame short orbit, and no other short orbits.
\end{lem}

The first two cases are easily dealt with.
\begin{lem}
\label{lem:max_noOne}
There is no $\mathbb{F}_{p^2}$ maximal curve $X$ of Type I of Lemma \ref{lem:worb_types}.
\end{lem}
\begin{proof}
The Hermitian curve is of Type IV, so by Proposition ~\ref{hermp^2} we can assume that a maximal $p$-subgroup of $X$ is of order $p$. Recall the notation $g = \frac{1}{2}c(p-1) + dp$. The size of the wild orbit is in the form $1 + np$ for some integer $n$ by \ref{prop:max_pgrp}. Letting $s$ denote the size of the stabilizer for the wild orbit, the Riemann-Hurwitz formula gives
\[\frac{2g - 2}{(1+np)s} = \frac{(c+1)(p-1) - 1}{s} = \frac{(c+1)(p-1) - 1}{s}.\]
Then
\[(1+np)(cp +p - c - 2) = 2g - 2.\]
If $n > 0$, the only case where genus is at most $\frac{1}{2}p(p-1)$ is $n = 1, c = 0$, and no case has genus strictly less than this. Since the Hermitian curve is not of this type, we have $n = 0$ and get
\[p = 2dp\]
which cannot happen.
\end{proof}

For Type II, we in fact have a result for $\mathbb{F}_{q^2}$ maximal curves instead of just $\mathbb{F}_{p^2}$ maximal curves.
\begin{prop}
There is no $\mathbb{F}_{q^2}$-maximal curve $X$ of Type II of Lemma \ref{lem:worb_types}.
\end{prop}
\begin{proof}
From Proposition \ref{prop:max_pgrp}, we see each wild orbit is of size congruent to one modulo $p$. One wild orbit must be a subset of the $\mathbb{F}_{q^2}$ rational points, since the number of such points is congruent to one modulo $p$. The number of remaining $\mathbb{F}_{q^2}$ rational points is then divisible by $p$, so the other wild orbit cannot lie among the $\mathbb{F}_{q^2}$ rational points. The other wild orbit must then lie in the set of $\mathbb{F}_{q^{2n}}$ rational points that are not $\mathbb{F}_{q^2}$ rational. However, these sets of points also have order divisible by $p$, so this is again impossible.
\end{proof}

The following lemma allows our divisibility-based methods to become precise at the end.
\begin{lem}
\label{lem:artin_types}
Let $X$ be a $\FFF_{p^2}$ maximal curve of genus at least two. If $X/H$ is rational for some subgroup $H\subset G$ of order $p$, then $X$ is isomorphic to $y^m = x^p - x$ for some $m$ dividing $p + 1$.
\end{lem}
\begin{proof}
The main theorem of \cite{Irok91} states that any Artin-Schreier curve with zero Hasse-Witt map must be of the form $y^m = x^p - x$ with $m$ dividing $p+1$. Since $\FFF_{p^2}$ maximal curves are superspecial, they by definition have zero Hasse-Witt map. Since $X/H$ is rational, $X$ is an Artin-Schreier curve, and the theorem applies.

\end{proof}

At this point in the proof of Theorem \ref{thm:max}, we can assume that there at most one short tame orbit and a unique wild orbit in the action of $G$ on $X$. There are three cases to consider. First, for relatively low genus curves, it is possible there is a free orbit among the $\FFF_{p^2}$ points. Otherwise, the $\FFF_{p^2}$ rational points either consist of just one wild orbit, a case dealt with in Lemma \ref{lem:max_wcase}, or are the union of a tame orbit with a wild orbit, a case dealt with in Lemma \ref{lem:max_wtcase}.
\begin{lem}
Suppose $X$ is a $\mathbb{F}_{p^2}$ maximal curve of genus $g\geq 2$ and $p > 7$. Suppose $X \rightarrow X/G$ is wildly ramified and that there is a free orbit among the $\FFF_{p^2}$ rational points. Then $|G| \le g^2$ or $|G| \le 84(g-1)$.
\end{lem}
\begin{proof}
Write $g = dp + \frac{1}{2}c(p-1)$ and write $np+1$ for the length of the wild orbit.

First, we deal with the case that $d > 0$ and $n > 1$. In this case, there must be a free orbit of size at most $(2d + c + 1)p^2 - (c+n)p$, which is  less than $g^2$ for $d > 2$, $d = 2$ and $c > 0$, or $d = 1$ and $c > 2$. We also have a lower bound of $p(np+1)$ on the size of the group. Then, if $|G| > g^2$, we must have
\[ 5p^2 - np \ge np^2 + p\]
so $n < 5$.

If the tame orbit is outside the $\FFF_{p^2}$ rational points, we must have the wild orbit size dividing the number of other points in $\FFF_{p^2}$. Then $np + 1 \bigm| (2d + c + 1)p - (c + n)$, so
\[np + 1 \bigm| n(n+c) + 2d + c + 1.\]
For $2 \le n \le 4$, $d > 0$, and $2d + c \le 4$, it is a quick computation to verify this only has solutions for $p = 2,3, 5, 7$.

Now, if the tame orbit is among the $\FFF_{p^2}$ rational points, we use Lemma \ref{lem:fine_cor} to say
\[np + 1 \bigm| 16(p+1)d(c+d+1)\]
which, from $p(n-1) = (np+1) - (p+1)$ gives
\[np+1 \bigm| 16(n-1)d(c+d+1).\]
For $2d + c \le 4$, we then have $np + 1 \bigm| 96(n-1)$ or $np + 1 \bigm| 64(n-1)$. Checking for $p$ with $n = 2, 3, 4$ is another easy computation. The only $p$ that can satisfy these congruences for $n$ in this range are $p = 2, 5$, and the case is done.

Now suppose that $n = 1$ with $d > 0$. In this case, the stabilizer of each point of the wild orbit is necessarily transitive on the other points, so that $G$ acts double transitively on the wild orbit. Take $\bar{G}$ to be the permutation group $G$ induces on these $p+1$ points. By element counting, no $\alpha \in \bar{G}$ that fixes two points of the orbit can have order more than $p-1$. But subgroups of the stabilizer of a point of size not dividing $p$ are necessarily cyclic, so the stabilizer of any pair of points is cyclic.

Then, by \cite{Kant72}, and since the $2$-transitive action is on a set of size $p+1$ and not $p^k + 1$ for some larger $k$, we know that $\bar{G}$ is either isomorphic to $\text{PSL}(2, p)$, $\text{PGL}(2, p)$, or has a regular normal subgroup. But $\text{PSL}(2, p)$ is of size $\frac{1}{2}(p^3 - p)$ and $\text{PGL}(2, p)$ is still larger, while $g < \frac{5}{2}p$ being necessary for $|G| > g^2$ implies the number of $\FFF_{p^2}$ points outside this orbit is at most $6p^2 - 2p$, so $p - 1 < 12$. Then exceptions can only exist for $p \le 11$. We can deal with the case $p = 11$ by noting $\text{PGL}(2, 11)$ is too large, and $\text{PSL}(2, 11)$ has size equal to $p^2 + 2gp - p$ only when $g$ is exactly $\frac{1}{2}5(p-1)$, outside the case $d \ne 0$.

If $\bar{G}$ instead has some regular normal subgroup, or a subgroup acting freely and transitively on the $p+1$ points, we know that the regular normal subgroup must be a elementary $p$-group. Excepting $p =2$ to $p + 1 = 3$, we then get that $p$ must be a Mersenne prime, of the form $2^k - 1$ for some $k$.

Taking an alternative tack, we can also write
\[\frac{2g-2}{rp(p+1)} = \frac{-k}{t} + \frac{(c+1)(p-1) - 1}{rp}\]
in this case for some positive integers $r, t$ and with $k$ either $0$ or $1$. Then
\[\frac{k|G|}{t} = p(c+1)(p-1) - 2dp\]
Then, for $d \le 2$, $k \ne 0$ for $p > 5$. But this tame orbit size must divide the number of $\FFF_{p^2}$ rational points that are not among the $p+1$, so
\[(c+1)(p-1) - 2d \bigm| (2d + c + 1)p - (c+1)\]
In fact, unless the tame orbit is among the $\FFF_{p^2}$ rational points, we see that $p+1 \bigm|(2d+c+1)p - (c+1)$ too, so $p+1 \bigm| 2d + 2c + 2$, which is impossible for $2d + c \le 5$, $d \ge 1$, and $p > 7$ (should $2d + c > 5$ with $d \ge 1$, $|G|$ is forced to be less than $g^2$). Then the tame orbit is $\FFF_{p^2}$ rational, so we also have that $(c+1)(p-1) - 2d$ divides $2(p^2 - 1)(p+1)$, so
\[(c+1)(p-1) - 2d \bigm| 32d(c + d + 1)^2\]
where the relation comes from consider the greatest common denominator of $p-1, p+1$ with the left hand side. This relation doubles as an inequality. If $c + 2d > 4$, $d > 0$, and $|G| > g^2$, we have from counting the number of points in the free orbit of $\FFF_{p^2}$ points that $p^2 - 14p + 9 < 0$ for $p < 17$. Otherwise, for $c + 2d \le 4$, $d > 0$, we get $p \le 9 \cdot 64 + 3$. The only Mersenne primes of interest are then $p = 3, 7, 31, 127$. Checking $31, 127$ for $(c, d) = (0, 1), (1, 1), (2, 1),  (0, 2)$, we see that they cannot be examples, so for $p > 7$ we again have no examples.

Now, suppose $d \ne 0$ and the wild orbit is of size $1$. In this case we can write
\[\frac{2g - 2}{rp} = \frac{-k}{t} + \frac{(c+1)(p-1) - 1}{rp}\]
for $k$ either $0$ or $1$. For $k = 0$, we get $2dp - p = 0$, which is impossible. For $k = 1$, we have
\[\frac{2dp - p}{rp} = \frac{-1}{t}\]
which is impossible for $d > 0$.

This just leaves all cases where $d = 0$. In this case, we know by Lemma \ref{lem:artin_types} that any possible curve is isomorphic to one of the form $y^m = x^p -  x$ with $m$ dividing $p+1$. But these do not have the specified type, as the automorphism group for each partitions the set of $\FFF_{p^2}$ rational points into either one or two short orbits. The lemma is proved.

\end{proof}

\begin{lem}
\label{lem:max_wcase}
For $p \ge 7$, the only $\FFF_{p^2}$ maximal curve of genus at least two with all of the $\FFF_{p^2}$ rational points forming a single wild orbit is the Hermitian curve.
\end{lem}
\begin{proof}
We first deal with the case where the maximal $p$-group has order $p$. Write $g = dp + \frac{1}{2}c(p-1)$, and write the size of the stabilizer for a point of the wild orbit as $rp$. Suppose $d \ne 0$. Then $r \le 4d + 2$, as $r$ is the order of a cyclic group in $X/H$, where $H$ is a $p$-group, and $X/H$ has genus $d > 0$. Write $p^2 + 2gp + 1 = N$. Assuming that there is at most one other tamely ramified orbit, we can write
\[\frac{2g - 2}{rpN} = \frac{-k}{t} + \frac{cp + p - c - 2}{rp}\]
But the magnitude of the right hand side, if nonzero, has a lower bound of
\[\frac{cp + p - c - 2}{(rp)(rp + 1)}\]
so
\[\frac{1}{rp^2} \ge \frac{2g-2}{2rgp^2} \ge \frac{2g-2}{rpN} \ge \frac{cp + p - c - 2}{(rp)(rp + 1)}\]
so $rp + 1 \ge cp^2 + p^2 - cp - 2p$, and so $(4d + 2)p + 1 \ge p(cp + p - c - 2)$, and finally
\[d \ge  \frac{1}{4}(cp^2 + p^2 - cp -4p - 1).\]
If $c = 0$, we get $d \ge \frac{1}{4}(p - 4)$. But then the genus is at least $\frac{1}{4}p(p-4)$ and is divisible by $p$, and no maximal curve fits this bill for $p \ge 7$ by checking against Theorem 10.48 of \cite{Hirs13}. If $c > 0$, the only curve in the range is the Hermitian curve.

Then we can assume $d = 0$. This case falls immediately to Lemma \ref{lem:artin_types}. The Hermitian curve is the only example of a curve of the form $y^m = x^p - x$, where $m$ divides $p+1$, whose automorphism group acts transitively on the $\FFF_{p^2}$ rational points.


\end{proof}

\begin{lem}
\label{lem:max_wtcase}
Let $X$ be a $\FFF_{p^2}$-maximal curve of genus at least two. Then $G$ has two short orbits in $X$, one wild and one tame, that have union equal to the set of $\FFF_{p^2}$ rational points if and only if $X$ is isomorphic to a curve given by the equation $y^m = x^p - x$, where $m \bigm| p+1$, $m \ne 1, p+1$.
\end{lem}
\begin{proof}
Write $1 + np$ for the size of the wild orbit and write $g = dp + \frac{1}{2}c(p-1)$. Then Riemann-Hurwitz gives
\[2g - 2 = -(p^2 + (2g- n)p) + (1 + np)((c+1)(p-1) - 1)\]
which reduces to
\[2d(p + 1) = (c+1)(n-1)(p-1)\]
But $p-1$ and $p+1$ have greatest common factor at most $2$, so $d \ge \frac{1}{4}(p-1)$ unless the right hand side of this expression vanishes. But the only $\FFF_{p^2}$ maximal curve with genus at least $\frac{1}{4}p(p-1)$ is the Hermitian curve, so the right hand vanishes. Then $n$ must be $1$ and $d$ must be zero.




By Lemma \ref{lem:artin_types}, we have that such a curve is of the form $y^m = x^p - x$ with $m$ dividing $p+1$. Taking away the cases $m = 1$ for being rational and $m = p+1$ for being Hermitian and not of this form, we have one direction of the lemma.

However, the curve $y^m = x^p - x$ with $m$ dividing $p+1$ is covered by the Hermitian curve. It is a twist of a maximal curve. Further, per Theorem 12.11 of \cite{Hirs13}, if $m \ne p+1$, the curve has two short orbits, one of size $p+1$ and the other of size $(c+1)p(p-1)$. These partition the rational points between them. For genus $g \ge 2$, we only need that $m > 1$. Then we have the converse.
\end{proof}

At this point we have exhausted all possibilities. Theorem \ref{thm:max} is true.

\subsection{Superspecial curves with irreducible canonical representations}

We can now prove Theorem \ref{thm:sup_irred}.

\begin{proof}
$X$ is isomorphic either to a $\FFF_{p^2}$-maximal or $\FFF_{p^2}$-minimal curve. In the latter case, since the genus of a $\FFF_{p^2}$ minimal curve is at most $\frac{1}{2}(p-1)$, as a curve cannot have negatively many points, we have that the only possible examples of an exception to $|G| \le 84(g-1)$ come from curves isomorphic to the Roquette curve from equation $y^2 = x^p - x$. Then $X$ can be assumed to be $\FFF_{p^2}$ maximal. But then, if $g > 82$ and $|G| > g^2$, then $|G| > 84(g-1)$ too. Further, if $p \le 7$, then we have $g \le 3 \cdot 7 = 21$. Then we can apply Theorem \ref{thm:max} to say that if $X$ has irreducible canonical representation, has $g > 82$, and is not isomorphic to the Roquette curve from $y^2 = x^p - x$, then it is isomorphic to $y^m = x^p - x$ for some $m \bigm| p+1$.
The following proposition then implies the theorem.

\begin{prop}
Let $X$ be the smooth projective model of $y^m=x^p-x$, where $m|p+1$. Then $X$ has an irreducible canonical representation if and only if $m\in\{2,p+1\}$.
\end{prop}

To prove this proposition, we first find a basis for $H^0(X,\Omega^1)$. Let $m'=\frac{p+1}{m}$. The other affine part of the curve will have equation $v^m=u-u^p$, where $u=\frac{1}{x}$ and $y=\frac{y}{x^{m'}}$. We have the following.
\begin{enumerate}
\item $x$ has a zero of order $m$ at $(x,y)=(0,0)$ and a pole of order $m$ at $\infty$ (there is only one point at $\infty$, corresponding to $(u,v)=(0,0)$), and no poles or zeroes anywhere else. Thus the divisor associated with $x$ is \[m\cdot (0,0)-m\cdot\infty.\]
\item $y$ has a zero of order $1$ at $(x,y)=(a,0)$ for all $a\in\mathbb{F}_p$, a pole of order $p$ at $\infty$, and no poles or zeroes anywhere else. Thus the divisor associated with $y$ is \[\displaystyle\sum_{a\in\mathbb{F}_p}(a,0)-p\cdot\infty.\]
\item $dx$ has a pole of order $m+1$ at $\infty$. Since $y$ has a zero of order $1$ at $(a,0)$ for $a\in\mathbb{F}_p$, $dy$ has no zeroes or poles there. Now, $dy=\frac{dx}{my^{m-1}}$ ($\gcd(m,p)=1$ as $m|p+1$), $dx$ must have a zero of order $m-1$ at all such points. Furthermore, it cannot have any other zeroes or poles, as shown by the equation $dy=\frac{dx}{my^{m-1}}$. Therefore, the canonical divisor associated to $dx$ is \[(m-1)\displaystyle\sum_{a\in\mathbb{F}_p}(a,0)-(m+1)\cdot\infty.\]
\end{enumerate}
The degree of the canonical divisor is thus $p(m-1)-(m+1)=pm-p-m-1$. By Riemann-Roch, this must equal $2\cdot g_X-2$, so $g_X=\frac{(p-1)(m-1)}{2}$.

Now, if $1\leq j\leq m-1$ and $0\leq i\leq\left\lfloor\frac{pj-1}{m}-1\right\rfloor$, then since the canonical divisor associated to $\frac{x^idx}{y^j}$ is
\[(m-1-j)\displaystyle\sum_{a\in\mathbb{F}_p}(a,0)+(pj-mi-(m+1))\cdot\infty,\]
we have that $\frac{x^idx}{y^j}\in H^0(X,\Omega^1)$ (since $m-1-j\geq 0$, $pj-mi-(m+1)\geq 0$).

Substituting $p=mm'-1$ into $i\leq\left\lfloor\frac{pj-1}{m}\right\rfloor-1$, we obtain
\begin{align*}
i & \leq\left\lfloor\frac{mm'j-j-1}{m}\right\rfloor-1 \\ & =m'j-1-\left\lceil\frac{j+1}{m}\right\rceil \\ & =m'j-2.
\end{align*}
Furthermore, it is clear that all of the $\frac{x^i dx}{y^j}$ are independent. The number of such elements with $1\leq j\leq m-1$ and $0\leq i\leq m'j-2$ is
\[m'\frac{m}{m-1}{2}-(m-1)=\frac{(p+1)(m-1)}{2}-(m-1)=\frac{(p-1)(m-1)}{2},\]
so these elements form a basis of $H^0(X,\Omega^1)$.

We now claim that the subspace generated by $\frac{x^i dx}{y}$, $0\leq i \leq m'-2$, is invariant under the automorphism group provided that $m<p+1$.  In this case, the automorphism group is generated by the following two transformations by \cite[Thm 12.11]{Hirs13}:
\begin{enumerate}
\item $\zeta_m$, an $m$th root of unity acting via $(x,y) \to (x,\zeta_m y)$.  
\item $\sigma = \begin{pmatrix} a & b \\ c & d \end{pmatrix} \in SL(2,\mathbb{F}_p)$ acting via $$(x,y)\to \left(\frac{ax+b}{cx+d},\frac{y}{(cx+d)^{m'}}\right).$$
\end{enumerate}

$\zeta_m$ visibly acts on the basis vectors simply by multiplication.  Finally, note that 
\begin{equation*} \sigma\left(\frac{x^idx}{y}\right) = (ax+b)^i(cx+d)^{m'-2-i}\frac{dx}{y}
\end{equation*}
is also a sum of the basis vectors of this subspace.  This is a proper subspace when $m>2$.  Thus, when $2<m<p+1$, the canonical representation is not irreducible.  When $m\in \{2,p+1\}$, the canonical representation is irreducible as stated in Proposition \ref{dh}.

\end{proof}

%% file: thm2case123.tex
\section{Automorphisms of Ordinary Curves}

In this section, we prove Theorem \ref{main}, improving upon a bound by Nakajima \cite{Naka87} on the automorphisms of ordinary curves.  The proof of the theorem borrows many techniques from the similar theorem in \cite{Naka87}, but improves the bound by using asymptotics and some more detailed arguments.  
\begin{proof}[Proof of Theorem \ref{main}]
Let $G= \text{Aut}(X)$ and $Y= X/G$. If $|G| > 84(g-1)$, then $Y$ has genus zero and Lemma \ref{lem:worb_types} applies. This splits the proof of the theorem into four cases, each corresponding to a type of curve from this Lemma.

\textbf{Case I}. This case has $p$ an odd prime, with one point $P \in Y$ wildly ramified, two points $Q, Q'$ tamely ramified with ramification index $2$, and no other ramified points. Choose $P_1 \in X$ over $P$. Then we can write $|G^{(0)}_{P'}| = Eq$, with $q$ a power of $p$ and $p \nmid E$. Hurwitz gives
\[2\left( \frac{g-1}{|G|}\right) = -2 + \frac{2}{2} + \frac{Eq + q - 2}{Eq} = \frac{q-2}{Eq}\]
But $Eq$ is the size of a subgroup of $G$, so we have
\[ 2\left(\frac{g-1}{q-2}\right)  = \frac{|G|}{Eq}\]
is an integer. Since $q-2$ is odd, we have $(q-2) | (g - 1)$. Write $a = \frac{g-1}{q-2}$. Also, $E$ divides $q-1$ by \cite[Proposition 1]{Naka87}, so write $q-1 = dE$. Then
\[|G| = 2aEq = \frac{2a}{d}\left(\frac{g-1}{a} + 1\right)\left(\frac{g-1}{a} + 2\right) = \frac{2(g + a - 1)(g + 2a - 1)}{ad}.\]
But $1 \le a < g$, so this equation forces $|G| \le 2(2g - 1)(g+1) \le 5g^2$, finishing this case.

\textbf{Case II}. In this case, there are two wildly ramified points in $Y$ and no other points. Renaming the points if necessary, we write $e_{Q_1} = E_1q$, $e_{Q_2} = E_2 q q'$ with $E_1, E_2$ integers not divisible by $p$. Then Hurwitz gives
\[2\left(\frac{g-1}{|G|}\right) = -2 + \frac{E_1q + q - 2}{E_1q} + \frac{E_2qq' + qq' -2}{E_2qq'} \]
If $qq' = q = 2$, this equation forces $X$ to be an elliptic curve, against  our assumption $g_X \ge 2$. We then assume $qq' > 2$.

Define $b_1$ and $b_2$ so $b_1E_1 = q - 1$, $b_2E_2 = qq' - 1$. $b_1$ and $b_2$ are positive integers by \cite[Proposition 1, Lemma 2]{Naka87}. We then rewrite the above equation as
\[\frac{|G|}{qq'} = \frac{2(g-1)(q-1)(qq'-1)}{b_2(q-1)(qq' - 2) + b_1(qq'-1)q'(q-2)}\]
Then
\[\frac{|G|}{qq'}  \le \frac{2(g-1)(qq'-1)}{b_2(qq'-2)}\]
with equality if and only if $q = 2$.
If $3 \le qq' \le 14$, this gives
\[|G| \le \frac{14 \cdot 13(g-1)}{6b_2}\]
which is less than $g^2$ for $g \ge 30$. Assuming $qq' \ge 15$ for the rest of this case, we have
\begin{equation}
\label{eq:15_7}
|G| \le \frac{15qq'(g-1)}{7b_2}
\end{equation}

We now split into three subcases. First, suppose $q \ne 2$ and $q' \ne 1$. $E_2qq'$ is the size of a subgroup of $G$, so
\[\frac{b_2|G|}{qq'(qq' - 1)} = \frac{2b_2(g-1)(q-1)}{b_2(q-1)(qq' - 2) + b_1(qq'-1)q'(q-2)}\]
is an integer. But the greatest common denominator of $q-1$ and the denominator will need to divide $b_1 (q' - 1)$, so
\[\frac{2b_2b_1(g-1)(q'-1)}{b_2(q-1)(qq' - 2) + b_1(qq'-1)q'(q-2)}\]
is also an integer. Then
\[\frac{2b_2(g-1)(q'-1)}{q'(qq' - 1)(q-2)} \ge 1\]
so $\frac{1}{b_2}(qq' - 1) \le 2(g-1)$ for $q > 2$. Then Equation \ref{eq:15_7} gives
\[|G| \le \frac{32}{7}(g-1)^2.\]

For the next subcase, take $q = 2$ and $q' \ne 1$.
In this case, our equation is
\[\frac{|G|}{2q'} = \frac{(g-1)(2q' - 1)}{b_2(q' - 1)}.\]
The left is an integer, so we find that $(q' - 1) | (g-1)$. Write $a(q' - 1) = (g-1)$. We find
\[|G| = 2aq'\frac{2q' - 1}{b_2} = \frac{1}{ab_2} 2(g + a - 1)(2g + a - 2)\]
which from $1 \le a < g$ gives $|G| \le 6g^2$.

If $q = 2$ and $q' = 1$, we would have $qq' = 2$, which we already dealt with.

Finally, suppose $q' = 1$. We have
\[\frac{b_2|G|}{q(q-1)} = \frac{2b_2(g-1)}{(b_1+b_2)(q-2)}\]
and
\[\frac{b_1|G|}{q(q-1)} = \frac{2b_1(g-1)}{(b_1+b_2)(q-2)}\]
are integers. Adding these, we see $a(q-2) = 2(g-1)$ for some $a$. Then
\[|G| = 2(g-1) \frac{q(q-1)}{(b_1 + b_2)(q-2)} \le \frac{15}{14}(g-1)2g \le 3g^2\]
where we are using $q = qq' \ge 15$. This finishes Case II.

\textbf{Case III}.  \cite{Naka87} shows that this case is impossible for ordinary curves.

%% file: thm2case4a.tex
\textbf{Case IV}.  Let $e_{Q_1} = Eq$ and $e_{Q_2} = e$, where $q=p^n$ and $(E,p)= (e,p) = 1$.  This implies that $\frac{d_{Q_1}}{e_{Q_1}} = \frac{Eq+q-2}{Eq}$ and $\frac{d_{Q_2}}{e_{Q_2}} = \frac{e-1}{e}$.  Applying the Riemann-Hurwitz formula, we have 

\begin{equation}
\label{hurwitz4}
 \frac{2g_X-2}{|G|} = \frac{(e-E)q-2e}{Eqe}.
 \end{equation}

By \cite[Lemmas 1,2]{Naka87}, it suffices to consider when $E > (\frac{g_X}{21})^{\frac{1}{2}}$, and thus, we take $E$ to be large in the following analysis.  Our strategy for the remainder of Case III is to first prove two bounds on $|G|$ in terms of $E,q,$ and $g_X$ in Lemma \ref{mainbounds}.  Then, we bound $E$ and $q$ in terms of $g_X$ to finish the proof.  

\begin{lem}\label{mainbounds}
If $E \ge 444$, we have the following estimates of $|G|$:
\begin{equation}
\label{lambda} |G|\leq 2(1+7E^{-1})Eq(g_X-1)
\end{equation}

\begin{equation}
\label{mu} |G| \leq 2(3+37E^{-1})E^2(g_X-1)
\end{equation}
\end{lem}

\begin{proof}
Define $d=(q-1)/E$ and $\varepsilon = e-E$.  It is shown in \cite[Proposition 1, Lemma 2]{Naka87} that $d, \varepsilon$ are positive integers.  We then define the variables 
$$\lambda = \frac{dE\varepsilon -2E -\varepsilon}{E+\varepsilon}$$ and $$\mu = \frac{dE\varepsilon-2E-\varepsilon}{(d+E^{-1})(E+\varepsilon)}$$ 

so that \eqref{hurwitz4} gives us 
\begin{equation}
\label{Glambdamu} |G| =\frac{2Eq}{\lambda}(g_X-1) = \frac{2E^2}{\mu}(g_X-1)
\end{equation}

First, we bound $\lambda$ from below.  If $\varepsilon \geq 4$, then certainly $\lambda \geq \frac{E\varepsilon}{E+\varepsilon} -2 >1$.  Otherwise, we have that $\lambda = (d\varepsilon-2) - (d\varepsilon -1)\frac{\varepsilon}{E+\varepsilon}$ and $\varepsilon \leq 3$.  But $\lambda >0$ implies that $d\varepsilon-2$ is a positive integer, and thus 
$$\lambda = (d\varepsilon-2) - (d\varepsilon -1)\frac{\varepsilon}{E+\varepsilon} \geq 1-\frac{6}{E}.$$  

Next, we claim that $\mu \geq \frac{1}{3} - \frac{4}{E}$.  As before, it is easy to see that if $\varepsilon \geq 3$, then we are done.  Thus, assume that $\varepsilon <3 $ and we have $$\mu = \frac{d\varepsilon -2}{d+E^{-1}} - \frac{\varepsilon(d\varepsilon -1)}{(d+E^{-1})(E+\varepsilon)} \geq \frac{d\varepsilon -2}{d} - \frac{4}{E}.$$

Since $\mu >0$, we have $d\varepsilon >2$.  Hence, this is minimized at $d=3$, $\varepsilon =1$, in which case we get our desired bound of $$\mu \geq \frac{1}{3} - \frac{4}{E}.$$  


Finally, we have the bounds

\begin{align*}
\frac{1}{\mu} &\leq \left(\frac{1}{3}- \frac{4}{E}\right)^{-1} \leq 3+\frac{37}{E}
\\ \frac{1}{\lambda} &\leq \left(1- \frac{6}{E}\right)^{-1} \leq 1+\frac{7}{E} \end{align*}

for all $E \ge 444$, and the result follows.

\end{proof}

Now we bound $E^2$ and $Eq$ by a multiple of $g_X$.  Let $P_1\in X$ be such that $\pi_{X/Y}(P_1) = Q_1.$  Let $Z=X/G_1(P_1)$ and $W=X/G_0(P_1)$.  We have a sequence of maps $X\to Z\to W$ such that $|\text{Gal}(X/Z)| = q$ and $W$ is the quotient of $Z$ by $G_0(P_1)/G_1(P_1)$, which is a cyclic group of order $E$ by Proposition \ref{ramprop}.  We now have two cases depending on whether or not $Z$ is rational.  

\begin{lem}Theorem \ref{main} holds in the case $g_Z\geq 1$.  
\end{lem}
\begin{proof}
Observing that the point $P_1$ is totally wildly ramified in $\pi_{X/Z}$, we have by the Riemann-Hurwitz formula: 
$$2g_X-2 \geq q(2g_Z-2) + 2(q-1)$$ which implies that $g_X\geq qg_Z$.  

We now bound $g_Z$ from below by a multiple of $E$ so that we may apply Lemma \ref{mainbounds} to get: 
\begin{equation}
\label{lem3bound}|G|\leq 2(1+7E^{-1})Eq(g_X-1) \leq 2(1+7E^{-1})\left(\frac{E}{g_Z}\right)g_X(g_X-1).
\end{equation}  To do so, we examine the quotient map $Z\to W = Z/C_E$ where $C_E \cong G_0(P_1)/G_1(P_1)$ denotes the cyclic group of order $E$.  Hereafter, for any positive integer $a|E$, we identify $C_a$ with the unique order $a$ cyclic subgroup of $C_E$.  Let the lengths of the short orbits of $C_E$ acting on $Z$ (that is, the orbits of size less than $E$) be $l_1, \cdots, l_{s+1}$.  We have that $l_i | E$.  We may assume that $l_{s+1} = 1$ corresponding to the image of the point $P_1$ in $Z$, which is fixed by all of $G_0(P_1)/G_1(P_1)$.  Applying the Riemann-Hurwitz formula to the map $Z\to W$ yields 
\begin{equation}\label{ZWhurwitz}
2g_Z-2 = E(2g_W-2) + \sum_{i=1}^{s+1} (E-l_i).\end{equation}

If $g_W>0$, then we have $2g_Z -2 \geq E-l_{s+1} = E-1$ which implies the result by (\ref{lem3bound}).  Thus, we assume $g_W=0$.  Then, we note that $2g_Z-2\geq 0$ and $\frac{E}{2} \leq  E- l_i \leq E-1$.  By \ref{ZWhurwitz}, we conclude that $s\geq 2$.  If $s\geq 4$, we have that 
\begin{align*}
2g_Z-2 = -2E + (E-1) + \sum_{i=1}^s(E-l_i) &\geq -2E + (E-1) + 4\left(\frac{E}{2}\right) 
\\ &= E-1 \end{align*}
and we are done as before.  Hence, we have two remaining cases, $s=2$ and $s=3$.  

\textbf{Subcase 1: $s=2$}. We have that $2g_Z-2 = E - l_1 - l_2 - 1$.  We let $a = \text{lcm}\left(\frac{E}{l_1}, \frac{E}{l_2}\right)$ and let $Z' = Z/C_a$.  Then, we observe that by the choice of $a$, the map $\pi_{Z'/W}$ is a map of degree $\frac{E}{a}$ which is only ramified at $\pi_{X/W}(P_1)$, where it is totally ramified.  It follows that by the Riemann-Hurwitz formula on $\pi_{Z'/W}$, we have that $2g_{Z'}-2 = -\frac{2E}{a} + (\frac{E}{a} - 1) = -\frac{E}{a} -1$.  But $g_{Z'} \geq 0$, so it follows that $a= E$ and therefore, $(l_1,l_2)=1$.  Assume without loss of generality that $l_1\geq l_2$.  It follows from the fact that $l_1$ and $l_2$ are relatively prime integers dividing $E$ that $l_1 l_2\leq E$.  Hence, we have that either $l_1+l_2 \leq \frac{E}{3} + 3$ or $l_1 = \frac{E}{2}$.  In the former case, we immediately have that $E\leq 3g_Z+3$ and thus $\frac{E}{g_Z} \leq 3(1-\frac{3}{E})^{-1} \leq 3(1+\frac{4}{E})$.  Thus, we have that by (\ref{lem3bound}), 
\begin{align}\label{s=2bound}
|G|\leq 6(1+12E^{-1})g_X(g_X-1)&\leq 6\left(1+12\sqrt{\frac{21}{g_X}}\right)g_X(g_X-1) 
\\\ &\leq 6(g_X^2 +12\sqrt{21}g_X^{\frac{3}{2}}) \nonumber \end{align}
 as desired.  

We now assume that $E$ is even and $l_1= \frac{E}{2}$.  It follows that $l_2$ is either $1$ or $2$.  We claim that $Z$ is not ordinary in both cases.  This would contradict the Deuring-\v{S}afarevi\v{c} formula applied to $\pi_{X/Z}$.    Let $Z'' = Z/C_2$ so that the map $\pi_{Z''/W}$ is a degree $\frac{E}{2}$ map with exactly two ramification points, $R_1 = \pi_{X/W}(P_1)$ and another point $R_2\in W$.  By the Riemann-Hurwitz formula applied to $\pi_{Z''/W}$, we have that $$2g_{Z''}-2 = \frac{E}{2}(-2)+2(E-1)$$ which implies that $g_{Z''} = 0$.  Therefore, since $\pi_{Z/Z''}$ has degree $2$, $Z$ is hyperelliptic.  In fact, since $g_{Z''} = g_W = 0$ and $\pi_{Z''/W}$ is a degree $\frac{E}{2}$ map which is totally ramified over $2$ points, we may assume up to automorphisms of $Z''$ and $W$ that $\pi_{Z''/W}$ is the map $z\mapsto z^{\frac{E}{2}}$.  Let $Z$ be the hyperelliptic curve given by $y^2 = f(x)$ such that the map $Z\to Z''$ is given by $(x,y)\mapsto x$.  Then, $\text{Gal}(Z''/W)\in \text{Aut}(Z'')$ is a cyclic group of order $\frac{E}{2}$ such that any element $\gamma \in \text{Gal}(Z''/W)$ fixes $0$ and $\infty$ and has the property that $(\gamma z)^{\frac{E}{2}} = z^{\frac{E}{2}}$.  Hence, $\gamma$ is multiplication by an $\frac{E}{2}$-th root of unity.  In fact, if $\gamma$ is chosen to be a generator of this cyclic group, then $\gamma$ is multiplication by a primitive $\frac{E}{2}$-th root of unity.  It follows that the map $Z\to Z'$ is branched over $0$, possibly $\infty$ depending on whether $l_2$ is $1$ or $2$, and $\frac{E}{2}$ other points which are permuted by $\text{Gal}(Z''/W)$.  Hence, these points must be related by multiplication by some $\frac{E}{2}$-th root of unity.  Thus, we have that $f(x) = x\left(x^\frac{E}{2}-a\right)$ and $Z$ is the hyperelliptic curve given by the equation $y^2 = x\left(x^{\frac{E}{2}}-a\right)$ for some $a\in k$.  We show that this curve is not ordinary for sufficiently large $E$.

This is a standard computation; we simply compute the action of Frobenius on a basis of $H^1(Z, \OO_Z)$ and show that the matrix of Frobenius is not invertible. For this hyperelliptic curve, a basis of $H^1(Z, \OO_Z)$ is given by $\{ \frac{y}{x^i} \}$ for $i=1,2,\cdots ,\left\lfloor\frac{E}{4}\right\rfloor$.  
If $p=2$, then Frobenius sends $\frac{y}{x^i}$ to $\frac{x^{\frac{E}{2}}-a}{x^{2i}}=0$ in $H^1(Z, \OO_Z)$, so in fact $Z$ is supersingular and we can ignore this case.  Otherwise, assume that $p>2$ and Frobenius sends
\begin{equation*}
\frac{y}{x^i} \mapsto \frac{y^p}{x^{pi}}=\frac{y\left(x\left(x^{\frac{E}{2}}-a\right)\right)^{\frac{p-1}{2}}}{x^{pi}}.
\end{equation*}
 Therefore, the $(i,j)$ entry of the matrix of Frobenius with respect to this basis, known as the \emph{Hasse-Witt matrix}, is
\[[pi-j]\left(x\left(x^{\frac{E}{2}}-a\right)\right)^{\frac{p-1}{2}}=\left[pi-j-\frac{p-1}{2}\right]\left(x^{\frac{E}{2}}-a\right)^{\frac{p-1}{2}},\]
where $[n]f(x)$ refers to the $x^n$ coefficient of $f(x)$. Thus the $(i,j)$ entry is nonzero if and only if $\frac{E}{2}|pi-j-\frac{p-1}{2}$, or equivalently, $pi-\frac{p-1}{2}\equiv j\pmod{\frac{E}{2}}$. 

Suppose $Z$ is ordinary. Then for each $i$, row $i$ must have some nonzero entry.  Thus, since $1\leq j\leq\left\lfloor\frac{E}{4}\right\rfloor$, we have that for each $i$, the smallest positive integer $n$ such that $$pi-\frac{p-1}{2}\equiv n \text{ } \left(\text{mod }\frac{E}{2} \right)$$ satisfies $1\leq n\leq \left\lfloor\frac{E}{4}\right\rfloor$.

Let $i=\left\lfloor\frac{E}{2p}\right\rfloor$.  First assume that $i\neq 0$.  Then, $pi-\frac{p-1}{2}$ is between $\frac{E}{2}-p+1-\frac{p-1}{2}$ and $\frac{E}{2}-\frac{p-1}{2}$.  Since its value modulo $\frac{E}{2}$ must be between $1$ and $\left\lfloor\frac{E}{4}\right\rfloor$, we have $\frac{E}{2}-p+1-\frac{p-1}{2}\leq\frac{E}{4}$, so $E\leq 6(p-1)$ and $g_X\leq 21E^2=756(p-1)^2$. Thus, $g_X$ is bounded by some $c=c(p)$ on the order of $p^2$ and we have finished this case.  Now, note that if $i=0$, we trivially have the same bounds.  

\textbf{Subcase 2: $s=3$}. We have that $2g_Z-2 = 2E - l_1 - l_2 - l_3 - 1$.  By the same argument as above, we obtain $(l_1,l_2,l_3) = 1$ and thus $l_1+l_2+l_3 \leq E + 2$ and we are done by the same bounds as in (\ref{s=2bound}).

\end{proof}

For the remainder of the proof, we assume that $g_Z=0$, which implies that $g_W=0$.  Applying the Riemann-Hurwitz formula to the cyclic cover $\pi_{Z/W}$ immediately yields that there must be exactly two ramification points, and they both must be totally ramified.  Let them be $R_1,R_2\in W$ such that $\pi_{X/W}(P_1)=R_1$.  We have the following lemma:

\begin{lem}If there exists $R_3\in W$, $R_3 \neq R_1, R_2$, which is ramified under $\pi_{X/W}$, then Theorem \ref{main} holds.  
\end{lem}
\begin{proof}
In $\pi_{X/W}$, $R_1$ is totally ramified by definition, so $d_{R_3} = Eq +q -2$ and $e_{R_3} = Eq$.  $R_2$ is totally ramified in $\pi_{Z/W}$, so we have that $e_{R_2} \geq E$ and thus $\frac{d_{R_2}}{e_{R_2}} \geq \frac{e_{R_2}-1}{e_{R_2}} \geq 1-\frac{1}{E}$.  Finally, since $R_3$ is unramified in $\pi_{Z/W}$, $e_{R_3}$ divides $q = \text{deg }\pi_{X/Z}$.  Hence, $d_{R_3} = 2e_{R_3}-2$ so we have that $\frac{d_{R_3}}{e_{R_3}} = 2 - \frac{2}{e_{R_3}}\geq 2-\frac{2}{p}$.  Finally, we apply the Riemann-Hurwitz formula to $\pi_{X/Z}$ which yields: 
$$\frac{2g_X-2}{Eq} \geq -2 + \frac{d_{R_1}}{e_{R_1}}+\frac{d_{R_2}}{e_{R_2}}+\frac{d_{R_3}}{e_{R_3}} \geq 2 -\frac{2}{Eq} - \frac{2}{p}.$$

Rearranging and using (\ref{lambda}), we have $$|G| \leq 2Eq(1+7E^{-1})(g_X-1) \leq 2\left(\frac{1+7E^{-1}}{1-\frac{1}{Eq}-\frac{1}{p}}\right)(g_X-1)^2.$$  The conclusion follows by letting $E$ grow large and using $p\geq 2$.  
\end{proof}

%% file: thm2case4b.tex
Now suppose that $R_1$ and $R_2$ are the only ramification points of $\pi_{X/W}$ in $W$. In this final case, we will show a stronger result; that is, that $|G|\leq g_X^2$ for $g_X$ sufficiently large, except possibly for one particular family of cases. As before, assume $E\geq 1$. Also assume $|G|>g_X^2$, and as usual we can and will also assume that $21E^2>g_X$.
\begin{lem}\label{dlargelemma}
In this case, $d\geq E-2$.
\end{lem}
\begin{proof}
Again letting $H=G_1(P_1)$, we have that $H$ is an elementary abelian group of order $q$. Choose some $P_2$ mapping to $R_2$. Let $N=H_0(P_2)=H_1(P_2)$. As in \cite{Naka87}, we let $q'=|N|$ and $q''=|H/N|$. Then $q'q''=q$.

By the argument in Lemma 5 of \cite{Naka87}, $E$ divides both $q'-1$ and $q''-1$, and $q'\neq 1$.
Therefore, either
\[q''=1\]
or
\[E^2\leq (q'-1)(q''-1)<q'q''=q.\]
In the former case, \cite[p. 606]{Naka87} shows that $E\leq\sqrt{q}+1$. This implies that $E^2-2E\leq q-1$, and so $d=\frac{q-1}{E}\geq E-2$, so Lemma \ref{dlargelemma} holds. In the latter case, $E^2\leq q-1$, which implies that $d=\frac{q-1}{E}\geq E$. Thus Lemma \ref{dlargelemma} holds here as well.
\end{proof}
\begin{lem}\label{ep3lemma}
If $\epsilon\geq 3$, then $g_X\leq 364$.
\end{lem}
\begin{proof}
\[\mu=\frac{dE\epsilon-2E-\epsilon}{(d+E^{-1})(E+\epsilon)}=\frac{E\epsilon}{E+\epsilon}\frac{d}{d+E^{-1}}-\frac{2E+\epsilon}{(d+E^{-1})(E+\epsilon)},\]
which is an increasing function in $d$. Therefore, we can substitute $d\geq E-2$ to obtain
\[\mu\geq\frac{(E-2)E\epsilon-2E-\epsilon}{(E-2+E^{-1})(E+\epsilon)}.\]
Using $2E\leq 2E\epsilon$ in the numerator and $E^{-1}\leq 1$ in the denominator, we have that
\[\mu\geq\frac{(E-4)E\epsilon-\epsilon}{(E-1)(E+\epsilon)}=\frac{\epsilon(E^2-4E-1)}{(E-1)(E+\epsilon)}.\]
This last expression is increasing in $\epsilon$, so if $\epsilon\geq 3$,
\[\mu\geq\frac{3(E^2-4E-1)}{E^2+2E-3}.\]
Substituting into the expression $|G|=\frac{2E^2}{\mu}(g_X-1)$,
\[|G|\leq\frac{2E^2(E^2+2E-3)}{3(E^2-4E-1)}(g_X-1)\]
(Assuming that this expression is positive, i.e. $E\geq 5$.)

Now, since $|G|>g_X^2$, this means
\[\frac{2E^2(E^2+2E-3)}{3(E^2-4E-1)}\geq\frac{|G|}{g_X-1}\geq\frac{g_X^2}{g_X-1}>g_X+1.\]
However, \cite{Naka87} also shows that if $q''\neq 1$, then $E^2<g_X$, and if $q''=1$, then $E\leq\sqrt{g_X+1}+1$. Either way, $g_X+1\geq (E-1)^2$. Therefore,
\[\frac{2E^2(E^2+2E-3)}{3(E^2-4E-1)}>(E-1)^2.\]
Therefore, $E\leq 20$.

Since $g_X+1<\frac{2E^2(E^2+2E-3)}{3(E^2-4E-1)}$ for $E\geq 5$, this means that if $5\leq E\leq 20$, $g_X+1<366$. Thus if $E\geq 5$, $g_X\leq 364$, so the lemma holds. If $E\leq 4$, we have the bound $E\geq\sqrt{\frac{g}{21}}$, so in this case, $g\leq 21\cdot 4^2=336$, so the lemma also holds here.
\end{proof}
We now consider the case $\epsilon\leq 2$. Recall that $\epsilon=e-E$. We rearrange the expression
\[\frac{2g_X-2}{|G|}=\frac{(e-E)q-2e}{Eqe}\]
into
\[\frac{|G|}{Eqe}=\frac{2g_X-2}{(e-E)q-2e}=\frac{2g_X-2}{dE\epsilon-2E-\epsilon}.\]
Now, note that since both $Eq$ and $e$ are the size of a stabilizer of a point in $X$ under $G$, $|G|$ is divisible by both $Eq$ and $e$. Now, $\gcd(q,e)=1$, as $p\nmid e$, and $\gcd(E,e)|e-E=\epsilon$. Therefore, $Eqe$ divides $\epsilon|G|$, and so
\[\frac{\epsilon(2g_X-2)}{dE\epsilon-2E-\epsilon}\]
is a positive integer.

We first consider the case $q''=1$. Then \cite{Naka87} shows that $g_X=q-1=dE$. Therefore,
\[\frac{2dE\epsilon-2\epsilon}{dE\epsilon-2E-\epsilon}\]
is an integer. Subtracting $2$,
\[\frac{4E}{dE\epsilon-2E-\epsilon}\]
is an integer. But $dE\epsilon-2E-\epsilon\geq dE-2E-1\geq E^2-4E-1>4E$ for $E\geq 9$. Therefore, $E\leq 9$. Thus $g_X<21E^2=1701$.

Now, suppose $q''\neq 1$. We have $E|q'-1$ and $E|q''-1$. Let $q'=p^{a'}$ and $q''=p^{a''}$. Then $E|\gcd{q'-1,q''-1}=p^{\gcd{a',a''}}-1$. Therefore, $E\leq p^{\gcd(a',a'')}-1$. Suppose $a'\neq a''$. Then $E^3<(p^{\gcd(a',a'')}-1)(p^{2\gcd(a',a'')}-1)\leq(q'-1)(q''-1)$, since one of $a'$ or $a''$ must be at least $2\gcd(a',a'')$. By the argument in Lemma $5$ of \cite{Naka87}, we have that $g_X\geq(q'-1)(q''-1)$. Therefore, $g_X\geq E^3$. But $21E^2>g_X$, so $E\leq 20$, implying that $g_X<21\cdot 20^2=8400$. Therefore, if $g_X\geq 8400$, then $q'=q''=\sqrt{q}$.

In this case, $\sqrt{q}=p^n$ for some positive integer $n$, and $q'=q''=p^n$. Then the argument in Lemma $5$ of \cite{Naka87} (applying the Hurwitz formula to the map $X\to Z$) shows that $g_X=p^{2n}-p^n$. Therefore,
\[\frac{\epsilon(2g_X-2)}{(\epsilon q-2e)}=\frac{2\epsilon(p^{2n}-p^n-1)}{\epsilon p^{2n}-2e}\]
is an integer, so subtracting $2$,
\[\frac{2(2e-\epsilon p^n-\epsilon)}{\epsilon p^{2n}-2e}\]
is an integer.
Now, since $E|q'-1=p^n-1$, $E\leq p^n-1$ and thus $e=E+\epsilon\leq p^n+1$, as $\epsilon\leq 2$. Therefore,
\[|2(2e-\epsilon p^n-\epsilon)|\leq\max(4e,2\epsilon p^n+2\epsilon)\leq 4p^n+4.\] Therefore, if $2e-\epsilon p^n-\epsilon\neq 0$,
\[4p^n+4\geq\epsilon p^{2n}-2e\geq p^{2n}-2p^n-2.\]
Therefore, $p^n<7$, so $p^n\leq 5$. Thus $E\leq p^n-1=4$, so $g_X\leq 21\cdot 4^2=336<8400$.

Therefore, if $g_X\geq 8400$, then $2e-\epsilon p^n-\epsilon=0$. This implies that $e=\frac{\epsilon}{2}(p^n+1)$, and thus $E=e-\epsilon=\frac{\epsilon}{2}(p^n-1)$. This also means that the integer
\[\frac{\epsilon|G|}{Eqe}=\frac{\epsilon(2g_X-2)}{(\epsilon q-2e)}\]
must in fact be equal to $2$. Thus if $\epsilon=1$, then $|G|=2Eqe=2\frac{p^n-1}{2}\cdot p^{2n}\cdot\frac{p^n+1}{2}=\frac{p^{4n}-p^{2n}}{2}$. But if $|G|>g_X^2=(p^{2n}-p^n)^2$, we have that
\[(p^{4n}-p^{2n}>2(p^{2n}-p^n)^2,\]
so $p^n=2$. But then $E=1$, a contradiction. Therefore, $\epsilon=2$. Thus $E=p^n-1$, and $e=p^n+1$. Since $\frac{2|G|}{Eqe}=2$, $|G|=Eqe=p^{4n}-p^{2n}$.

Therefore, if $g_X\geq 8400$ and $|G|>g_X^2$, then $E=p^n-1$, $e=p^n+1$, $q=p^{2n}$, $g_X=p^{2n}-p^n$, and $|G|=p^{4n}-p^{2n}$ for some positive integer $n$. Theorem \ref{main} trivially follows from this stronger result in this case.
\end{proof}